\documentclass[12pt,reqno]{amsart}

\usepackage{amsmath,amssymb,amsthm,mathrsfs}
\usepackage{graphicx,cite,times}
\usepackage{cases}
\usepackage{color}
\usepackage{appendix}
\usepackage{enumerate}

\usepackage[colorlinks, linkcolor=blue, citecolor=blue]{hyperref}
\usepackage{cite}

\newtheorem{theo}{Theorem}[section]
\newtheorem{coro}[theo]{Corollary}
\newtheorem{lemm}[theo]{Lemma}

\newtheorem{rema}[theo]{Remark}

\numberwithin{equation}{section}

\begin{document}

\title[Stability for the  perturbed fourth-order Schr\"{o}dinger equation]
{Stability of inverse boundary value problem for the fourth-order Schr\"{o}dinger equation}

\author{Yang Liu}
\address{School of Mathematics and Statistics,  Northeast Normal University, Changchun,
Jilin 130024, China}
\email{liuy694@nenu.edu.cn}

\author{Yixian Gao}
\address{School of
Mathematics and Statistics, Center for Mathematics and
Interdisciplinary Sciences, Northeast Normal University, Changchun, Jilin 130024, China}
\email{gaoyx643@nenu.edu.cn}

\thanks{ The research of YL was supported by NSFC grant 12401554 and Excellent Youth Project of Jilin Provincial Department of Education JJKH20250294K.
The research of YG was supported by NSFC grant 12371187 and Science and Technology Development Plan Project of Jilin Province 20240101006JJ
}

\subjclass[2010]{35B35, 35R30, 35J40}
\keywords{perturbed fourth-order Schr\"{o}dinger equation, Cauchy data, complex geometric optics solution, stability}

\begin{abstract}
This paper is concerned with the stability of the  inverse boundary value problem for the  perturbed fourth-order Schr\"{o}dinger equation in a bounded domain with  Cauchy data.
We establish stability results for the perturbed potential relying on boundary measurements.
The estimates depend on various a priori information regarding the regularity  and  the support of the inhomogeneity.
The proof primarily utilizes the complex geometric optics solution method and  Fourier analysis.
\end{abstract}

\maketitle

\section{introduction and main results}

This paper aims to study the stability of the inverse boundary value problem for the  perturbed fourth-order Schr\"{o}dinger equation.
Let $\Omega$ denote a  bounded open set in $\mathbb R^3$ with boundary $\partial \Omega$ smooth enough.
The perturbed fourth-order Schr\"{o}dinger equation with the Navier boundary conditions is given by
\begin{align}\label{fourth order equation*}
\begin{cases}
\Delta^2u(k, \boldsymbol x)+\gamma \Delta u(k, \boldsymbol x)-k^4u(k, \boldsymbol x)+q(\boldsymbol x)u(k, \boldsymbol x)=0, ~\quad &\boldsymbol x\in\Omega,\\
u(k, \boldsymbol x)=f_1(k, \boldsymbol x), ~\quad\Delta u(k, \boldsymbol x)=f_2(k, \boldsymbol x),~\quad &\boldsymbol x\in \partial\Omega,
\end{cases}
\end{align}
where  $\gamma\in\mathbb R$ is a parameter that accounts for possible lower-order dispersion and $k>0$ is the wave number.
Without loss of generality, we may assume that $\Omega$ is contained within a unit ball, and the potential  $q(\boldsymbol x)\in L^\infty(\mathbb R^3)$ satisfies $\text{supp}~q(\boldsymbol x)\subset \Omega$.

The fourth-order Schr\"{o}dinger equation  arises in many scientific fields, such as quantum mechanics,
condensed matter physics, and optical physics.
It is a natural extension and development of the second-order Schr\"{o}dinger operator.
Compared with the latter, the scattering theory of the former still requires further exploration and refinement.
The fourth-order equation was first proposed  with  a small fourth-order dispersion term to describe
the propagation of intense laser beams in a bulk medium with Kerr nonlinearity  \cite{karpman, karp2}.
Wave phenomena related to this equation include optical waveguides in optics and optical solitons in light, among others.
From a mathematical perspective, some important properties of the fourth-order {Schr\"{o}dinger} equation can refer to \cite{Fibich2002}.
The direct problem has been studied  by using harmonic analysis and the energy method \cite{Pausader2007, Pausader2009}.
 Additionally, in \cite{Miao2009}, the authors have shown global well-posedness for nonlinear {S}chr\"{o}dinger equations of fourth-order in the radial case. In the nonlinear case, the blowup of the solution is determined by  $\gamma$, especially the equation \eqref{fourth order equation*} with $\gamma=0$  has scaling invariance \cite{Boulenger2017}.

 Determining the potential or medium  for the  inverse scattering problem for acoustic, electromagnetic, and elastic waves has aroused the interest of physicists, engineers, and applied mathematicians, and it has significant applications in various scientific areas.
Most studies  in the literature are devoted to the inverse scattering problem for acoustic wave equations, Schr\"{o}dinger equations and Maxwell equations, see \cite{Hahner2001, Liu2015, liu2016, knox2020, Lijingzhi2021, Deng2019, Deng20192, wang2021, isakov2011}.
 Unlike  the second order partial differential operators, the fourth-order Schr\"{o}dinger operator is more complicated.
 Some literature has focused on  the uniqueness and stability for fourth-order elliptic operators.
 For uniqueness results, see, for instance \cite{G2010, Gunther2012, Katsi2014, Yang2014, Yang2017,GLL2023}.
 If $\gamma$ is equal to zero, the stability results for the source or the potential can be found in  \cite{Liuboya2020, li2021stability}.
The stability estimate for the source with the damped term has been established in  \cite{Li2021}.
{However, the presence of the perturbation term $\gamma$ forces the stability bound to depend on $\gamma$;  consequently the stability estimate is affected compared with the unperturbed case.}
Therefore, we aim to derive  a stability result for the potential and obtain  a sharp estimate  that depends on the coefficient $\gamma$.

Due to the lack of well-posedness for the problem \eqref{fourth order equation*}, we utilize Cauchy data as measurement data.
Cauchy data sets are typically used to solve inverse boundary value problems.
The advantage of this method is that it does not require proving the well-posedness of the direct scattering problem.
This idea, mentioned in \cite{isa2016, nagayasu2013}, has been used to determine the conductivity and potential.

The inverse problem for the  perturbed fourth-order Schr\"{o}dinger equation \eqref{fourth order equation*} can be described as follows: to  determine the potential by knowing the  boundary data.
The corresponding stability estimates mainly depend on boundary measurements,  which can be represented by the Cauchy data.

The Cauchy data set for the boundary value problem \eqref{fourth order equation*} is defined as
\begin{align*}
C_{q}:=\bigg\{(u|_{\partial\Omega}&, \Delta u|_{\partial\Omega}, \partial_\nu u|_{\partial\Omega}, \partial_\nu(\Delta u)|_{\partial\Omega})\big|~u\in H^4(\Omega),~\Delta^2u+\gamma \Delta u-k^4u+q(\boldsymbol x)u=0\bigg\},
\end{align*}
where $\nu$ is the exterior unit normal vector to $\partial\Omega$.
 The distance between the different sets of Cauchy data is given by
\begin{align*}
{\rm dist}(C_{q_{1}}, C_{q_{2}}):=\max\Bigg\{ \underset{h_{u_1}\in C_{q_{1}}}{\sup} &\underset{h_{u_2}\in C_{q_{2}}}{\inf}\frac{\|h_{u_1}-h_{u_2}\|_{H^{7/2, 3/2, 5/2, 1/2}(\partial\Omega)}}{\|h_{u_1}\|_{H^{7/2, 3/2, 5/2, 1/2}(\partial\Omega)}}, \\
& \underset{h_{u_2}\in C_{q_{2}}}{\sup} \underset{h_{u_1}\in C_{q_{1}}}{\inf}\frac{\|h_{u_2}-h_{u_1}\|_{H^{7/2, 3/2, 5/2, 1/2}(\partial\Omega)}}{\|h_{u_2}\|_{H^{7/2, 3/2, 5/2, 1/2}(\partial\Omega)}}\Bigg\}
\end{align*}
with the norm
\begin{align*}
	\|h_{u}\|_{H^{7/2, 3/2, 5/2, 1/2}(\partial\Omega)}=\left(\|u\|^2_{H^{7/2}(\partial\Omega)}+\|\Delta u\|^2_{H^{3/2}(\partial\Omega)}+\|\partial_\nu u\|^2_{H^{5/2}(\partial\Omega)}+\|\partial_\nu(\Delta u)\|^2_{H^{1/2}(\partial\Omega)}\right)^{1/2}.
\end{align*}
Note that $H^s(\Omega)$, $s>0$, denotes the usual Sobolev space with the norm defined by
\begin{align*}
\|u\|_{H^s(\Omega)}:=&\left(\int_{\mathbb R^3}(1+|\boldsymbol \xi|^2)^s|\hat {u}(\boldsymbol\xi)|^2~{\rm d}\boldsymbol \xi\right)^{1/2},
\end{align*}
where $\hat {u}$ is the Fourier transform of $u$.
One advantage of choosing the Cauchy data is that it avoids the need to discuss
the well-posedness of the direct scattering problem, allowing us to focus on the inverse problem.
{This inverse problem can be formulated without assuming that $0$ is not a Dirichlet eigenvalue by using the framework of Cauchy data sets. Indeed, when $0$ is not a Dirichlet eigenvalue for $\Delta^2+\gamma \Delta -k^4+q$ in $\Omega$, the problem shows that knowing the Cauchy data set $C_q$ is equivalent to knowing the Dirichlet-to-Neumann map \cite{feldman2019calderon}.
Furthermore, the uniqueness result for the first-order perturbation $\gamma$ is established directly using the Cauchy data set in \cite{Gunther2012}.}

\subsection{Statement of the main results}
Assume that there exists a constant $c_s >0$ such that
 the potential function set satisfies:
\begin{align*}
\mathscr Q:=\left\{q(\boldsymbol x)>0: \|q(\boldsymbol x)\|_{H^s(\mathbb R^3)}\leq c_s, ~\text{for some fixed}~ s>3/2, ~\text{and}~\text{the constant}~c_s>0 \right\}.	
\end{align*}

\begin{theo}\label{stability}
	Suppose that $q_i(\boldsymbol x)\in \mathscr Q\cap L^\infty(\mathbb R^3), i=1, 2, $ and ${\rm dist}(C_{q_{1}}, C_{q_{2}})$ is {sufficiently small}.
	Then there exists a constant $C_3$ such that the following estimate holds
	\begin{align*}
	\|(q_1-q_2)(\boldsymbol x)\|_{L^2(\Omega)}\leq &C_3\big((-\ln ({\rm dist}(C_{q_{1}}, C_{q_{2}})))^{-s/(s+3)}\\
	&+(1+|\gamma|)(\gamma^2+4k^4)^2{\rm dist}(C_{q_{1}}, C_{q_{2}})^{(s+2)/(s+3)}\big),
\end{align*}
where $C_3$ depends on  $s, c_s$ and $\Omega$.
\end{theo}
{The stability estimate is hybrid in nature, comprising  a logarithmic term and a dominant H\"older term.
When $k$ is small, the logarithmic term dominates.
As the wavenumber $k$ increases, the H\"older component becomes predominant.
For $k\rightarrow \infty$, the H\"older term prevails and yields better stability than the logarithmic term.}
\begin{rema}
The potential term $q(\boldsymbol x)$ can also be considered  as a nonlinear term $V(\boldsymbol x, u(\boldsymbol x))$ for the nonlinear fourth-order Schr\"{o}dinger equation.
Compared with the former, we can consider the nonlinear term $V(\boldsymbol x, u(\boldsymbol x))=\lambda(\boldsymbol x)|u|^\alpha u$  (see \cite{Liu20214, Pausader2009, Fibich2002}).
The stability for  $\lambda(\boldsymbol x)$ can be established by using linearization techniques.
\end{rema}

\begin{coro}\label{sta2}
	Under the assumptions in Theorem \ref{stability},  we have the estimate
	\begin{align*}
	\|(q_1-q_2)(\boldsymbol x)\|_{L^\infty(\Omega)}
	\leq &C_7\big((-\ln({\rm dist}(C_{q_{1}}, C_{q_{2}})))^{-(2s-3)/(2s+3)}\\
	&+(1+|\gamma|)(\gamma^2+4k^4)^2{\rm dist}(C_{q_{1}}, C_{q_{2}})^{{(2s+2)/(2s+3)}}\big),
\end{align*}
where $C_7$ depends on $s, c_s$ and $\Omega$.
\end{coro}

We modify the a priori information to be
\begin{align*}
\tilde{\mathscr{Q}}:=\{q(\boldsymbol x)>0: \|q(\boldsymbol x)\|_{W^{m, 1}(\mathbb R^3)}\leq c_m, ~\text{for some fixed}~m>3, \text{ and the constant}~ c_m>0\},
\end{align*}
where the norm of the Sobolev space $W^{m, 1}(\mathbb R^3)$ is defined by
\begin{align*}
\|u\|_{W^{m, 1}(\mathbb R^3)}:=\int_{\mathbb R^3}(1+|\boldsymbol \xi|^2)^{m/2}|\hat {u}(\boldsymbol\xi)|~{\rm d}\boldsymbol \xi.
\end{align*}

\begin{theo}\label{reducesta3}
Suppose that $q_i(\boldsymbol x)\in \tilde {\mathscr{Q}} \cap L^\infty(\mathbb R^3)$ for $ i=1, 2, $ and ${\rm dist}(C_{q_{1}}, C_{q_{2}})$ is {sufficiently small}.
Then we have the stability estimate
	\begin{align*}
	\|(q_1-q_2)(\boldsymbol x)\|_{L^\infty(\Omega)}
	\leq &C_{10}\big((-\ln({\rm dist}(C_{q_{1}}, C_{q_{2}})))^{-(m-3)/3}\\
	&+(1+|\gamma|)(\gamma^2+4k^4)^2{\rm dist}(C_{q_{1}}, C_{q_{2}})^{2/3}\big),
	\end{align*}
	where $C_{10}$ depends on $m, c_m$ and $\Omega$.
 \end{theo}

Note that above the  positive constants  $C_3, C_7,$ and $ C_{10}$ can be referred to subsection \ref{proof}.
The proof of our main results proceeds as follows. First, we construct complex geometric optics (CGO) solutions for the perturbed fourth-order Schr\"{o}dinger equation. Using boundary measurements and these CGO solutions, we derive an integral inequality relating the difference between potentials $\gamma_1$ and $\gamma_2$ to the difference in their corresponding Cauchy data. Departing from the conventional    ``cut the low frequencies last" strategy, we instead separate the inequality into low-frequency and high-frequency components. The resulting estimates depend on various a priori assumptions concerning the regularity and support of the inhomogeneity. Complete technical details are provided in Section \ref{3} and \ref{4}.


\section{The CGO solution for the  perturbed fourth-order Schr\"{o}dinger equation}\label{3}
In this section, we will  construct the   complex geometric optics  (CGO) solution for the  perturbed fourth-order Schr\"{o}dinger equation
\begin{align}\label{de}
\Delta^2u+\gamma\Delta u-k^4u+q(\boldsymbol x)u=0~\quad \text{in}~~ \Omega.
\end{align}
Obviously, if $q(\boldsymbol x)\equiv 0$ in $\Omega$, we find that $u_0(\boldsymbol x) = e^{{\rm i}\boldsymbol \theta\cdot \boldsymbol x}$ is a solution of
 \begin{align*}
 	\Delta^2 u_0+\gamma\Delta u_0-k^4 u_0=0 ~\quad \text{in}~\Omega,
 \end{align*}
where  the complex vector $ \boldsymbol \theta \in \mathbb C^3$ satisfies
\[\boldsymbol \theta\cdot\boldsymbol \theta=\frac{\sqrt{\gamma^2+4k^4}+\gamma}{2}.\]
Then, the form
\[u(\boldsymbol x) = e^{{\rm i}\boldsymbol \theta\cdot \boldsymbol x}(1+ p(\boldsymbol x))\]
is a solution of \eqref{de} if and only if   $p(\boldsymbol x)$  satisfies the following modified Faddeev type equation
\begin{align}\label{faddeev}
\Delta^2_\theta p+q(\boldsymbol x) p=-q(\boldsymbol x)~\quad \text{in}~\Omega,
\end{align}
where
\begin{align}
\Delta^2_\theta p:=&\Delta ^2 p+4{\rm i}\boldsymbol \theta\cdot\nabla\Delta p-2(\boldsymbol\theta\cdot\boldsymbol\theta)\Delta p-4(\nabla\nabla p\cdot\boldsymbol\theta)\cdot\boldsymbol\theta\nonumber\\
&-4{\rm i}(\boldsymbol\theta\cdot\boldsymbol\theta)(\boldsymbol \theta\cdot\nabla p)+\gamma\Delta p+2{\rm i}\gamma\nabla p\cdot \boldsymbol \theta.\label{square Fa}
\end{align}
To verify that $p$ is a solution of \eqref{faddeev},
we extend the domain from   the bounded domain  $\Omega\subset \mathbb R^3$ to  a cube $C_{\mathcal R}=[-\mathcal R, \mathcal R]^3$ with $\mathcal R>0$.
Define a grid
\begin{align}\label{grid}
\Gamma:=\left\{\boldsymbol \iota=(\iota_1, \iota_2, \iota_3)^\top\in \mathbb R^3:~\frac{\mathcal R}{\pi}\iota_1\in \mathbb Z, ~\frac{\mathcal R}{\pi}\iota_2-\frac{1}{2}\in \mathbb Z, ~\frac{\mathcal R}{\pi}\iota_3\in \mathbb Z \right\},	
\end{align}
and let  $e_{\boldsymbol \iota}(\boldsymbol x)=(2\mathcal R)^{-3/2}e^{{\rm i}\boldsymbol \iota\cdot \boldsymbol x}$ for $ \boldsymbol x\in C_{\mathcal R}$ and $ \boldsymbol \iota\in \Gamma$.
It is easy to see that $\{e_{\boldsymbol \iota}(\boldsymbol x)\}_{\boldsymbol\iota\in \Gamma}$  is an orthonormal basis in $L^2(C_{\mathcal R})$.
Additionally,  the orthonormal basis $\{e_{\boldsymbol \iota}(\boldsymbol x)\}_{\boldsymbol\iota\in \Gamma}$ is  complete, i.e.,
if $v\in L^2(C_{\mathcal R})$ satisfies
$
	(ve^{{\rm i}\pi/(2\mathcal R)x_2}, e^{{\rm i}\pi\boldsymbol n/\mathcal R\cdot \boldsymbol x})_{L^2(C_{\mathcal R})}=0,
$
$\boldsymbol n \in \mathbb Z^3$, then $(v, e_{\boldsymbol \iota})_{L^2(C_{\mathcal R})}=0$ for all $\boldsymbol \iota\in \Gamma$  implies $v=0$.

\begin{lemm}\label{homosolu*}
Let $\boldsymbol \theta \in \mathbb C^3$, and assume that the imaginary part of $\boldsymbol \theta$ satisfies
\[|{\rm Im}~\boldsymbol\theta|\geq\max\{1, (\sqrt{\gamma^2+4k^4}+\gamma)/2\},\]
and
\[\boldsymbol \theta\cdot\boldsymbol \theta=\frac{\sqrt{\gamma^2+4k^4}+\gamma}{2}.\]
Then, for any $g(\boldsymbol x)\in L^2(\Omega)$, there exists a solution $p\in H^4(\Omega)$ satisfying
\begin{align*}
	\Delta^2_\theta p(\boldsymbol x)=g(\boldsymbol x)~\quad \text{in}~\Omega,
\end{align*}
and the following estimate holds:
\begin{align*}
\|D^{\alpha} p\|_{L^2(\Omega)}&\leq C|{\rm Im}~\boldsymbol \theta|^{\alpha-1},~\quad \alpha=0, 1, 2, 3, 4,
\end{align*}
 where  the operator $\Delta^2_\theta$ is given in \eqref{square Fa}, and $C$ is a suitable constant.
\end{lemm}
\begin{proof}
It follows from
\begin{align*}
\boldsymbol \theta\cdot\boldsymbol \theta=|{\rm Re}~\boldsymbol \theta|^2-|{\rm Im}~\boldsymbol \theta|^2+2{\rm i}{\rm Re}~\boldsymbol \theta\cdot{\rm Im}~\boldsymbol \theta=\frac{\sqrt{\gamma^2+4k^4}+\gamma}{2}
\end{align*}
that $|{\rm Re}~\boldsymbol \theta|$ and  $|{\rm Im}~\boldsymbol \theta|$ satisfy
\begin{align*}
|{\rm Re}~\boldsymbol \theta|^2-|{\rm Im}~\boldsymbol \theta|^2=\frac{\sqrt{\gamma^2+4k^4}+\gamma}{2}.
\end{align*}
Then, by rotating coordinates (orthogonal transformation) in a suitable way, we can assume that ${\rm Re}~\boldsymbol \theta=(|{\rm Re}~\boldsymbol \theta|, 0, 0)^\top$, ${\rm Im}~\boldsymbol \theta=(0, |{\rm Im}~\boldsymbol \theta|, 0)^\top$ (see e.g., \cite{feldman2019calderon}).

As demonstrated in \cite{Hahner1996}, we adopt the same approach to prove the  existence of $p$: for any function $g(\boldsymbol x)\in L^2(\Omega)$, we prove  that there exists a solution $p(\boldsymbol x)\in H^4(\Omega)$ to the  equation
\begin{align*}
	\Delta^2_\theta p=g ~\quad \text{in}~\Omega.
\end{align*}
We extend $ g\in L^2(\Omega)$ by zero outside $\Omega$ into $C_{\mathcal R}$, denote it by $\tilde{g}$.
Using Fourier series in a shifted lattice with the orthonormal basis $\{e_{\boldsymbol \iota}(\boldsymbol x)\}_{\boldsymbol\iota\in \Gamma}$,
 we can express $\tilde{g}\in L^2(C_\mathcal R)$ as
\begin{align*}
\tilde{g}(\boldsymbol x)=\underset{\boldsymbol \iota\in\Gamma}{\sum}\hat {\tilde{g}}_{\boldsymbol \iota} e_{\boldsymbol \iota}(\boldsymbol x),
\end{align*}
where the Fourier coefficients are given by  $\hat{\tilde{g}}_{\boldsymbol \iota}:=(\tilde g, e_{\boldsymbol \iota})_{L^2(C_{\mathcal R})}$.
  Assume that  the  solution  takes  the form $p=\underset{\boldsymbol \iota\in\Gamma}{\sum}\hat {p}_{\boldsymbol \iota} e_{\boldsymbol \iota}(\boldsymbol x)$, such that
for any $\tilde g\in L^2(C_{\mathcal R})$, the equation
\begin{align}\label{homo13}
\Delta^2_\theta p=\tilde g
\end{align}
is satisfied.
Substituting  $p=\underset{\boldsymbol \iota\in\Gamma}{\sum}\hat {p}_{\boldsymbol \iota} e_{\boldsymbol \iota}(\boldsymbol x)$ into \eqref{homo13}, we obtain
 \begin{align}\label{fourierse3}
 \mathcal W_\iota\hat {p}_{\boldsymbol \iota} =\hat {\tilde{g}}_{\boldsymbol \iota},
 \end{align}
where
 \begin{align*}
 	\mathcal W_\iota=&|\boldsymbol\iota|^4+4(\boldsymbol \theta\cdot\boldsymbol \theta)(\boldsymbol \theta\cdot\boldsymbol \iota)+2|\boldsymbol\iota|^2(\boldsymbol \theta\cdot\boldsymbol \theta)\\
 	&+4(\boldsymbol\theta\cdot\boldsymbol\iota)^2+4|\boldsymbol\iota|^2(\boldsymbol \theta\cdot\boldsymbol \iota)-\gamma|\boldsymbol\iota|^2-2\gamma(\boldsymbol\theta\cdot\boldsymbol\iota)\\
 	=&\big(|\boldsymbol\iota|^2+2(\boldsymbol \theta\cdot\boldsymbol \iota)\big)^2+2(\boldsymbol \theta\cdot\boldsymbol \theta)\big(|\boldsymbol\iota|^2+2(\boldsymbol \theta\cdot\boldsymbol \iota)\big)\\
 	&-\gamma\big(|\boldsymbol\iota|^2+2(\boldsymbol \theta\cdot\boldsymbol \iota)\big)\\
 	=&\big(|\boldsymbol\iota|^2+2(\boldsymbol \theta\cdot\boldsymbol \iota)+2(\boldsymbol \theta\cdot\boldsymbol \theta)-\gamma\big)\big(|\boldsymbol\iota|^2+2(\boldsymbol \theta\cdot\boldsymbol \iota)\big).
 \end{align*}
Denoting
\begin{align*}
\mathcal M_{\boldsymbol \iota}=&(\boldsymbol\iota\cdot	\boldsymbol\iota)+2(\boldsymbol \theta\cdot\boldsymbol \iota)=|\boldsymbol\iota|^2+2|{\rm Re}~\boldsymbol \theta|\iota_1+2{\rm i}|{\rm Im}~\boldsymbol \theta|\iota_2,
 \end{align*}
 we can express $\mathcal W_{\boldsymbol \iota}$ as
\begin{align*}
\mathcal W_{\boldsymbol \iota}=\big(\mathcal M_{\boldsymbol \iota}+2(\boldsymbol \theta\cdot\boldsymbol \theta)-\gamma\big)\mathcal M_{\boldsymbol \iota}.
\end{align*}
Since $\boldsymbol \theta\cdot\boldsymbol \theta=\frac{\sqrt{\gamma^2+4k^4}+\gamma}{2}$, we can get
\begin{align*}
{\rm Im}~(\mathcal M_{\boldsymbol \iota}+2(\boldsymbol \theta\cdot\boldsymbol \theta)-\gamma)={\rm Im}~\mathcal M_{\boldsymbol \iota}=2|{\rm Im}\boldsymbol \theta|\iota_2.
\end{align*}
Thus we have
\begin{align*}
 |{\rm Im}~(\mathcal M_{\boldsymbol \iota}+2(\boldsymbol \theta\cdot\boldsymbol \theta)-\gamma){\rm Im}~\mathcal M_{\boldsymbol \iota}|=4|{\rm Im}~\boldsymbol \theta|^2\iota_2^2.
\end{align*}
Additionally, it is easy to verify that
 \begin{align*}
 |(a+{\rm i}b)(a+c+{\rm i}b)|^2=&	(a(a+c)-b^2)^2+b^2(2a+c)^2\\
=&a^4+a^2c^2+b^4-2a^2b^2-2ab^2c+2a^3c+4a^2b^2+4ab^2c+b^2c^2\\
=&a^4+2a^3c+a^2c^2+b^4+2a^2b^2+2ab^2c+b^2c^2\\
=&a^2(a+c)^2+b^2(a+c)^2+a^2b^2+b^4\\
\geq &b^4=|{\rm Im}(a+{\rm i}b){\rm Im}(a+c+{\rm i}b)|^2.
 \end{align*}
Then, by $|{\rm Im} ~\boldsymbol \theta|\geq 1$ and \eqref{grid}, we have
\begin{align*}
	|\mathcal W_{\boldsymbol \iota}|\geq& |{\rm Im}~(\mathcal M_{\boldsymbol \iota}+2(\boldsymbol \theta\cdot\boldsymbol \theta)-\gamma){\rm Im}~\mathcal M_{\boldsymbol \iota}|\\
	=&4|{\rm Im}~\boldsymbol \theta|^2\iota_2^2\geq\frac{\pi^2}{\mathcal R^2}|{\rm Im}~\boldsymbol \theta|.
\end{align*}
It follows from  \eqref{fourierse3} that
\begin{align}\label{inverse13}
	|\hat {p}_{\boldsymbol \iota}|=\frac{1}{|\mathcal W_\iota|}{|\hat {\tilde{g}}_{\boldsymbol \iota}|}\leq \frac{C}{|{\rm Im}~\boldsymbol \theta|}|\hat {\tilde{g}}_{\boldsymbol\iota}|.
\end{align}
Therefore, for any $\tilde{g}\in L^2(C_{\mathcal R})$ , the series $\underset{\boldsymbol \iota\in\Gamma}{\sum}\hat {p}_{\boldsymbol \iota} e_{\boldsymbol \iota}(\boldsymbol x)$ with $\hat {p}_{\boldsymbol \iota}$  given by \eqref{inverse13} converges  to a function $p(\boldsymbol x)$ in $L^2(C_{\mathcal R})$.
Accordingly, we deduce
\begin{align*}
\|p\|_{L^2(\Omega)}= (\underset{\boldsymbol \iota\in\Gamma}{\sum}|\hat {p}_{\boldsymbol \iota}|^2)^{\frac{1}{2}}\leq C(\underset{\boldsymbol \iota\in\Gamma}{\sum}\frac{1}{|{\rm Im}~\boldsymbol \theta|^2}|\hat {\tilde{g}}_{\boldsymbol \iota}|^2)^{\frac{1}{2}}=\frac{C}{|{\rm Im}~\boldsymbol \theta|}\|g\|_{L^2(\Omega)}.
\end{align*}
Taking the  derivative of $p$ with respect to $x_h, h=1, 2, 3$, we have
\begin{align*}
	\partial_{x_h} p=\underset{\boldsymbol{\iota}\in\Gamma}{\sum}{\rm i}\iota_h \hat{p}_{\boldsymbol \iota}e_{\boldsymbol \iota}, ~\quad h=1, 2, 3.
\end{align*}
By the estimate \eqref{inverse13}, we have
\begin{align}\label{buchong1}
   |\iota_h\hat p_{\boldsymbol \iota} |\leq |\iota_h||\hat{p}_{\boldsymbol \iota}|\leq C\frac{|\boldsymbol \iota|}{|{\rm Im}~\boldsymbol \theta|}|\hat {\tilde{g}}_{\boldsymbol\iota}|.
\end{align}
\begin{enumerate}[(i)]
	\item~For $|\boldsymbol\iota|\leq 8\sqrt 2 |{\rm Im}~\boldsymbol \theta|$,  it follows from \eqref{buchong1} that
	\begin{align*}
		\|\partial_{x_h} p\|_{L^2(\Omega)}=\|\underset{\boldsymbol{\iota}\in\Gamma}{\sum}{\rm i}\iota_h \hat{p}_{\boldsymbol \iota}e_{\boldsymbol \iota}\|_{L^2(\Omega)}\leq C\|g\|_{L^2(\Omega)}, ~\quad h=1, 2, 3.
	\end{align*}
	\item~For $|\boldsymbol\iota|>8\sqrt 2 |{\rm Im}~\boldsymbol \theta|$,
	because
	\[|{\rm Re}~\boldsymbol \theta|^2=\frac{\sqrt{\gamma^2+4k^4}+\gamma}{2}+|{\rm Im}~\boldsymbol \theta|^2\leq 2|{\rm Im}~\boldsymbol \theta|^2,\]
	 we have
	\begin{align*}
	   |\mathcal W_{\boldsymbol \iota}|\geq& \big|{\rm Re}~(\mathcal M_{\boldsymbol \iota}+2(\boldsymbol \theta\cdot\boldsymbol\theta)-\gamma){\rm Re}~\mathcal M_{\boldsymbol \iota}|\\
=&|(|\boldsymbol\iota|^2+2|{\rm Re}~\boldsymbol \theta|\iota_1+2(\boldsymbol \theta\cdot\boldsymbol\theta)-\gamma)(|\boldsymbol\iota|^2+2|{\rm Re}~\boldsymbol \theta|\iota_1)|\\
	   =&||\boldsymbol\iota|^4+4|\boldsymbol\iota|^2|{\rm Re}~\boldsymbol \theta|\iota_1+4|{\rm Re}~\boldsymbol \theta|^2\iota_1^2+2(\boldsymbol \theta\cdot\boldsymbol\theta)|\boldsymbol \iota|^2\\
	   &+4(\boldsymbol \theta\cdot\boldsymbol\theta)|{\rm Re}~\boldsymbol \theta|\iota_1-\gamma|\boldsymbol\iota|^2-2\gamma|{\rm Re}~\boldsymbol \theta||\iota_1|\\
	   \geq &|\boldsymbol\iota|^4-4|\boldsymbol \iota|^2|{\rm Re}~\boldsymbol \theta||\boldsymbol \iota|+2(\boldsymbol \theta\cdot\boldsymbol\theta)|\boldsymbol \iota|^2-4(\boldsymbol \theta\cdot\boldsymbol\theta)|{\rm Re}~\boldsymbol \theta||\boldsymbol \iota|-\gamma|\boldsymbol\iota|^2-2\gamma|\boldsymbol \iota||{\rm Re}~\boldsymbol \theta|\\
	   \geq &|\boldsymbol\iota|^4-4\sqrt 2|\boldsymbol \iota|^3|{\rm Im}~\boldsymbol \theta|+\sqrt{\gamma^2+4k^4}|\boldsymbol\iota|^2\\
	  & -2\sqrt 2(\sqrt{\gamma^2+4k^4}+\gamma)|\boldsymbol \iota||{\rm Im}~\boldsymbol \theta|-2\sqrt 2\gamma|\boldsymbol \iota||{\rm Im}~\boldsymbol \theta|\\
	   \geq&\frac{|\boldsymbol\iota|^4}{2}+(\frac{3}{4}\sqrt{\gamma^2+4k^4}-\frac{\gamma}{2})|\boldsymbol\iota|^2\\
	   \geq&\frac{|\boldsymbol\iota|^4}{2}.
	\end{align*}
Hence, we can obtain
	\begin{align*}
   |\iota_h\hat p_{\boldsymbol \iota} |\leq |\iota_h||\hat{p}_{\boldsymbol \iota}|\leq C\frac{|\iota_h|}{|\mathcal W_{\boldsymbol \iota}|}|\hat {\tilde {g}}_{\boldsymbol\iota}|&\leq \frac{2C}{|\boldsymbol\iota|^3}|\hat {\tilde{g}}_{\boldsymbol\iota}|\leq C|\hat {\tilde{g}}_{\boldsymbol\iota}|,
\end{align*}
\end{enumerate}
then we derive
\begin{align*}
	\|\partial_{x_h} p\|_{L^2(\Omega)}\leq C\|g\|_{L^2(\Omega)}, ~\quad h=1, 2, 3.
\end{align*}
Furthermore, taking the  derivative of $\partial_{x_h}p$ with respect to $x_m, m=1, 2, 3$ again, we have
\begin{align*}
\partial_{x_m} \partial_{x_h}p=\underset{\boldsymbol{\iota}\in\Gamma}{\sum}{\rm i}\iota_m{\rm i}\iota_h \hat{p}_{\boldsymbol \iota}e_{\boldsymbol \iota}, ~\quad h, m=1, 2, 3.	
\end{align*}
Repeating the above process, for $|\boldsymbol\iota|\leq8\sqrt 2 |{\rm Im}~\boldsymbol \theta|$, we get
\begin{align*}
	\|\partial_{x_mx_h} p\|_{L^2(\Omega)}\leq C|{\rm Im}~\boldsymbol\theta|\|g\|_{L^2(\Omega)}, ~\quad  h, m=1, 2, 3.
\end{align*}
For $|\boldsymbol\iota|>8\sqrt 2|{\rm Im}~\boldsymbol \theta|$,  we have
	\begin{align*}
	|{\rm i}\iota_m{\rm i}\iota_h \hat{p}_{\boldsymbol \iota}|\leq \frac{2C|\boldsymbol\iota|^2}{|\boldsymbol\iota|^4}|\hat {\tilde{g}}_{\boldsymbol\iota}|\leq \frac{2C}{|\boldsymbol\iota|^2}|\hat {\tilde{g}}_{\boldsymbol\iota}|\leq C|{\rm Im}~\boldsymbol \theta||\hat {\tilde{g}}_{\boldsymbol\iota}|,
	\end{align*}
which implies
	\begin{align}\label{reminder 24}
		\|\partial_{x_mx_h} p\|_{L^2(\Omega)}\leq C|{\rm Im}~\boldsymbol \theta|\|g\|_{L^2(\Omega)}, ~\quad h, m=1, 2, 3.
	\end{align}
The following estimates are similar to the proof of \eqref{reminder 24}.
For $|\boldsymbol\iota|\leq8\sqrt 2 |{\rm Im}~\boldsymbol \theta|$, it is easy to note that
\begin{align*}
	\|\partial_{x_nx_mx_h} p\|_{L^2(\Omega)}\leq &C|{\rm Im}~\boldsymbol\theta|^2\|g\|_{L^2(\Omega)}, ~\quad h, m, n=1, 2, 3,\\
	\|\partial_{x_px_nx_mx_h} p\|_{L^2(\Omega)}\leq &C|{\rm Im}~\boldsymbol\theta|^3\|g\|_{L^2(\Omega)}, ~\quad h, m, n, p=1, 2, 3.
\end{align*}
For $|\boldsymbol\iota|>8\sqrt 2 |{\rm Im}~\boldsymbol \theta|$, it gives
\begin{align*}
	|{\rm i}\iota_n{\rm i}\iota_m{\rm i}\iota_h \hat{p}_{\boldsymbol \iota}|\leq \frac{2C|\boldsymbol\iota|^3}{|\boldsymbol\iota|^4}|\hat {\tilde{g}}_{\boldsymbol\iota}|, ~\quad |{\rm i}\iota_p{\rm i}\iota_n{\rm i}\iota_m{\rm i}\iota_h\hat{p}_{\boldsymbol \iota}|\leq \frac{2C|\boldsymbol\iota|^4}{|\boldsymbol\iota|^4}|\hat {\tilde{g}}_{\boldsymbol\iota}|.
\end{align*}
From the above  estimates,  we conclude
\begin{align}\label{zongin}
\|D^\alpha p\|_{L^2(\Omega)}\leq C|{\rm Im}~\boldsymbol \theta|^{\alpha-1}\|g\|_{L^2(\Omega)}, ~\quad\alpha=0, 1, 2, 3, 4.
\end{align}
\end{proof}

\begin{lemm}\label{homosolu}
If $\boldsymbol \theta \in \mathbb C^3$  satisfies
\[\boldsymbol \theta\cdot\boldsymbol \theta=\frac{\sqrt{\gamma^2+4k^4}+\gamma}{2}\]
and  the imaginary part of $\boldsymbol \theta$ satisfies
\begin{align}\label{restriction}
	|{\rm Im}~\boldsymbol\theta|\geq\max\{1, ~(\sqrt{\gamma^2+4k^4}+\gamma)/2, ~2C\|q(\boldsymbol x)\|_{L^\infty(\Omega)}\},
\end{align}
then there exists a solution $p\in H^4(\Omega)$ satisfying
\begin{align}\label{homofaddeev}
	\Delta^2_\theta p+q (\boldsymbol x) p=-q(\boldsymbol x)~\quad \text{in}~\Omega,
\end{align}
and the following estimate holds:
\begin{align}
\|D^{\alpha} p\|_{L^2(\Omega)}&\leq C|{\rm Im}~\boldsymbol \theta|^{\alpha-1},~\quad \alpha=0, 1, 2, 3, 4, \label{ine2}
\end{align}
 where  the operator $\Delta^2_\theta$ is given in \eqref{square Fa}, and $C$ is a suitable constant.
\end{lemm}
\begin{proof}
For any $g\in L^2(\Omega)$, it is  sufficient to show that the extension solution to the  equation
\begin{align}\label{homofaddeev2}
 	\Delta^2_\theta p+q(\boldsymbol x)p=g ~\quad \text{in}~\Omega
 \end{align}
 exists.
If $p$ is a solution of \eqref{homofaddeev2} of the form
\begin{align*}
p=(\Delta^2_\theta)^{-1}\mathscr G,
\end{align*}
then the function $\mathscr G\in L^2(\Omega)$ needs to be determined.
Substituting  $p=(\Delta^2_\theta)^{-1}\mathscr G$ into \eqref{homofaddeev2}, we obtain
\begin{align}\label{equa2}
(I+q(\boldsymbol x) (\Delta^2_\theta)^{-1})\mathscr G=g.
\end{align}
It follows from \eqref{zongin} and \eqref{restriction} that
\begin{align*}
\|q(\boldsymbol x) (\Delta^2_\theta)^{-1}\|	_{L^2(\Omega)\rightarrow L^2(\Omega)}\leq C\frac{\|q(\boldsymbol x)\|_{L^\infty(\Omega)}}{|{\rm Im}~\boldsymbol\theta|}\leq \frac{1}{2}.
\end{align*}
This ensures the existence of   $(I+q(\boldsymbol x)(\Delta_\theta^2)^{-1})^{-1}$, which implies that  $\mathscr G=(I+q(\boldsymbol x)(\Delta_\theta^2)^{-1})^{-1}g$ is a solution of  \eqref{equa2} and satisfies
\[\|\mathscr G\|_{L^2(\Omega)}\leq 2\|g\|_{L^2(\Omega)}.\]
As a result, the function
\begin{align*}
p=(\Delta_\theta^2)^{-1}(I+q(\boldsymbol x)(\Delta_\theta^2)^{-1})^{-1}g
\end{align*}
is a solution of \eqref{homofaddeev2} and  satisfies \eqref{zongin}.
Recalling the equation \eqref{homofaddeev}, substituting $g=-q(\boldsymbol x)\in L^2(\Omega)$ into \eqref{zongin}, there exists $p(\boldsymbol x)\in H^4(\Omega)$ such that $u=e^{{\rm i}\boldsymbol \theta\cdot\boldsymbol x}(1+p(\boldsymbol x))$ is a solution of \eqref{faddeev}.
 This completes the proof.
\end{proof}	
%
\section{Stability estimates for the potential}\label{4}
In this section, we discuss the stability estimates for the potential $q(\boldsymbol x)$ in $H^s(\mathbb R^3)$ for some fixed $s>3/2$.
Furthermore, we establish an optimized stability exponent for  $q(\boldsymbol x)\in W^{m, 1}(\mathbb R^3)$ with $m>3$.

\subsection{An important inequality}
\begin{lemm}\label{substr}
Suppose that,
$q_i(\boldsymbol x)\in L^\infty(\mathbb R^3), i=1, 2,$ and  $ u_i\in H^4(\Omega), i=1, 2$
are  solutions of
\begin{align*}
\Delta^2u_i+\gamma\Delta u_i-k^4u_i+q_i(\boldsymbol x)u_i=0~\quad \text{in}~~ \Omega.
 \end{align*}
 Then, the following estimate holds
\begin{align}\label{ineq13}
\left|\int_\Omega (q_1(\boldsymbol x)-q_2(\boldsymbol x)) u_1  u_2~{\rm d}\boldsymbol x \right|\leq C(1+|\gamma|)\|u_1\|_{H^4(\Omega)}{\rm dist}(C_{q_{1}}, C_{q_{2}})\|u_2\|_{H^4(\Omega)},
\end{align}
where $C$ is a suitable constant.
\end{lemm}
\begin{proof}
Applying  Green's formula
\begin{align*}
\int_\Omega (\Delta^2 u) v-u(\Delta^2v)~{\rm d}\boldsymbol x=\int_{\partial \Omega}\partial_\nu(\Delta u)v-\Delta u(\partial_\nu v)-u\partial_\nu(\Delta v)+(\partial_\nu u)(\Delta v)~{\rm d}S,
\end{align*}
we have
	\begin{align*}
		0=&\int_\Omega( \Delta^2u_1+\gamma\Delta u_1-k^4u_1+q_1(\boldsymbol x)u_1)u_2-u_1(\Delta^2u_2+\gamma\Delta u_2-k^4u_2+q_2(\boldsymbol x)u_2)~{\rm d}\boldsymbol x \\
		=&\int_\Omega (q_1(\boldsymbol x)-q_2(\boldsymbol x)) u_1 u_2~{\rm d}\boldsymbol x+\gamma\int_{\partial\Omega}\partial_{\nu}u_1 u_2-u_1\partial_\nu u_2-\partial_{\nu}u_1 u_1+\partial_{\nu}u_1 u_1~{\rm d}S\\
		&+\int_{\partial\Omega}\partial_{\nu}(\Delta u_1) u_2+\partial_\nu u_1(\Delta u_2)- u_1\partial_\nu (\Delta u_2)-(\Delta u_1)\partial_\nu u_2~{\rm d}S\\
		&-\int_{\partial\Omega}\partial_{\nu}(\Delta u_1) u_1-\partial_\nu u_1(\Delta u_1)+u_1\partial_\nu (\Delta u_1)+(\Delta u_1)\partial_\nu u_1~{\rm d}S.
	\end{align*}
It is easy to see that
\begin{align*}
\int_\Omega&(q_1(\boldsymbol x)-q_2(\boldsymbol x)) u_1  u_2~{\rm d}\boldsymbol x\nonumber\\
=&-\int_{\partial\Omega}\partial_{\nu}(\Delta u_1)(u_2-u_1)+\partial_\nu u_1(\Delta u_2-\Delta u_1)~{\rm d}S\\
&+ \int_{\partial\Omega}u_1(\partial_\nu (\Delta u_2)-\partial_\nu(\Delta u_1))+(\Delta u_1)(\partial_\nu u_2-\partial_\nu u_1)~{\rm d}S\\
&-\gamma\int_{\partial\Omega}\partial_{\nu}u_1 (u_2-u_1)-u_1(\partial_\nu u_2-\partial_\nu u_1)~{\rm d}S.
\end{align*}
It follows from the Cauchy-Schwartz inequality that
\begin{align*}
	\bigg|\int_\Omega& (q_1(\boldsymbol x)-q_2(\boldsymbol x)) u_1 u_2~{\rm d}\boldsymbol x\bigg|\\
	=&\bigg|\int_{\partial\Omega}\partial_{\nu}(\Delta u_1)( u_2- u_1)+\partial_\nu u_1(\Delta u_2- \Delta u_1)\\
	&- u_1(\partial_\nu (\Delta u_2)- \partial_\nu (\Delta u_1))-(\Delta u_1)(\partial_\nu  u_2- \partial_\nu  u_1)~{\rm d}S\bigg|\\
	&+|\gamma|\bigg|\int_{\partial\Omega}\partial_{\nu}u_1 (u_2-u_1)-u_1(\partial_\nu u_2-\partial_\nu u_1)~{\rm d}S\bigg|\\
	\leq &4\left(\|u_1\|^2_{H^{7/2}(\partial\Omega)}
	+\|\Delta u_1\|^2_{H^{3/2}(\partial\Omega)}+\|\partial_\nu u_1\|^2_{H^{5/2}(\partial\Omega)}
	+\|\partial_{\nu}(\Delta u_1)\|^2_{H^{1/2}(\partial\Omega)}\right)^{1/2}\\
	&\cdot\underset{h_{u_1}\in C_{q_{1}}}{\inf} \Bigg\{\|u_2-u_1\|^2_{H^{7/2}(\partial\Omega)}+\|\Delta u_2-\Delta u_1\|^2_{H^{3/2}(\partial\Omega)}\\
	&+\|\partial_\nu u_2-\partial_\nu  u_1\|^2_{H^{5/2}(\partial\Omega)}+\|\partial_\nu (\Delta u_2)-\partial_\nu (\Delta u_1)\|^2_{H^{1/2}(\partial\Omega)}\Bigg\}^{1/2}\\
&
	+2|\gamma|(\|u_1\|^2_{H^{7/2}(\partial\Omega)}
+\|\partial_\nu  u_1\|^2_{H^{5/2}(\partial\Omega)})^{1/2}\\
&\cdot\underset{h_{u_1}\in C_{q_{1}}}{\inf} \Bigg\{\|u_2-u_1\|^2_{H^{7/2}(\partial\Omega)}+\|\partial_\nu u_2-\partial_\nu  u_1\|^2_{H^{5/2}(\partial\Omega)}\Bigg\}^{1/2}\\
	\leq &(4+2|\gamma|)\|h_{u_1}\|_{H^{7/2, 3/2, 5/2, 1/2}(\partial\Omega)}\frac{\underset{h_{u_1}\in C_{q_{1}}}{\inf}  \|h_{u_2}-h_{u_1}\|_{H^{7/2, 3/2, 5/2, 1/2}(\partial\Omega)}}{\|h_{u_2}\|_{H^{7/2, 3/2, 5/2, 1/2}(\partial\Omega)}}\|h_{u_2}\|_{H^{7/2, 3/2, 5/2, 1/2}(\partial\Omega)}\\
	\leq &C(1+|\gamma|)\|h_{u_1}\|_{H^{7/2, 3/2, 5/2, 1/2}(\partial\Omega)}{\rm dist}(C_{q_{1}}, C_{q_{2}})\|h_{u_2}\|_{H^{7/2, 3/2, 5/2, 1/2}(\partial\Omega)}.
\end{align*}
From  the Trace Theorem 5.1.7 and Theorem 5.1.9 in \cite{2015Agranovich}, we have
\begin{align*}
\|h_{u_i}\|_{H^{7/2, 3/2, 5/2, 1/2}(\partial\Omega)}=&(\|u_i\|^2_{H^{7/2}(\partial\Omega)}
	+\|\Delta u_i\|^2_{H^{3/2}(\partial\Omega)}+\|\partial_\nu u_i\|^2_{H^{5/2}(\partial\Omega)}
	+\|\partial_{\nu}(\Delta u_i)\|^2_{H^{1/2}(\partial\Omega)})^{1/2}\\
		\leq& c_1(\|u_i\|^2_{H^4(\Omega)}+\|\Delta u_i\|^2_{H^{2}(\Omega)}+\|u_i\|^2_{H^4(\Omega)}+\|\Delta u_i\|^2_{H^{2}(\Omega)})^{1/2}\\
		\leq&c_2\|u_i\|_{H^4(\Omega)}, ~\quad i=1,2.
\end{align*}
The proof is finished.
\end{proof}
\subsection{The stability results}\label{proof}
In this subsection, we will provide  the detail proof process of the main results.
\subsection*{The proof of Theorem \ref{stability}}
By the Fourier expansion, we have
\begin{align*}
(q_1-q_2)(\boldsymbol x)=\frac{1}{(2\mathcal R)^{3/2}}\underset{\boldsymbol\iota\in \Gamma}{\sum} \widehat{(q_{1\boldsymbol\iota}-q_{2\boldsymbol\iota})}e^{{\rm i}\boldsymbol\iota\cdot\boldsymbol x}.
\end{align*}
Let $\eta> 2$, then one has
\begin{align}\label{ineq15}
	|(q_1-q_2)(\boldsymbol x)|^2&=|\underset{\boldsymbol\iota\in \Gamma}{\sum} \widehat{(q_{1\boldsymbol\iota}-q_{2\boldsymbol\iota})}e^{{\rm i}\boldsymbol\iota\cdot\boldsymbol x}|^2=\underset{|\boldsymbol\iota|\in \Gamma}{\sum} |\widehat{(q_{1\boldsymbol\iota}-q_{2\boldsymbol\iota})}|^2\nonumber\\
	&=\underset{|\boldsymbol\iota|>\eta}{\sum} |\widehat{(q_{1\boldsymbol\iota}-q_{2\boldsymbol\iota})}|^2+\underset{|\boldsymbol\iota|\leq\eta}{\sum} |\widehat{(q_{1\boldsymbol\iota}-q_{2\boldsymbol\iota})}|^2=I_1+I_2.
\end{align}
To solve $I_1$, with the priori information for density, we immediately obtain
	\begin{align}\label{I1}
		I_1=&\underset{|\boldsymbol\iota|>\eta}{\sum}\frac{1}{(1+\boldsymbol \iota\cdot\boldsymbol \iota)^s}(1+\boldsymbol \iota\cdot\boldsymbol \iota)^s |\widehat{(q_{1\boldsymbol\iota}-q_{2\boldsymbol\iota})}|^2\nonumber\\
		\leq &\frac{1}{(1+\eta^2)^s}\underset{|\boldsymbol\iota|>\eta}{\sum}(1+\boldsymbol \iota\cdot\boldsymbol \iota)^s |\widehat{(q_{1\boldsymbol\iota}-q_{2\boldsymbol\iota})}|^2\nonumber\\
		\leq &\frac{c_s^2}{(1+\eta^2)^s}	\leq \frac{c_s^2}{\eta^{2s}}.
	\end{align}
Suppose that
\[\beta\geq\max\{1, ~(\sqrt{\gamma^2+4k^4}+\gamma)/2, ~2C\|q(\boldsymbol x)\|_{L^\infty(\Omega)}\},\]
 and  choose unit vectors $\boldsymbol y_1, \boldsymbol y_2\in\mathbb R^3$, which satisfy $\boldsymbol y_1\cdot \boldsymbol y_2=\boldsymbol y_1\cdot\boldsymbol \iota= \boldsymbol y_2\cdot\boldsymbol\iota=0$.
Define
\begin{align*}
	\boldsymbol{\theta}_1=&-\frac{\boldsymbol \iota}{2}+\sqrt{\frac{(\sqrt{\gamma^2+4k^4}+\gamma)}{2}-\frac{|\boldsymbol \iota|^2}{4}+\beta^2}\boldsymbol y_1+{\rm i}\beta\boldsymbol y_2\in\mathbb C^3,\\
	 \boldsymbol{\theta}_2=&-\frac{\boldsymbol \iota}{2}-\sqrt{\frac{(\sqrt{\gamma^2+4k^4}+\gamma)}{2}-\frac{|\boldsymbol \iota|^2}{4}+\beta^2}\boldsymbol y_1-{\rm i}\beta\boldsymbol y_2\in\mathbb C^3.
\end{align*}
It is clear that $|{\rm Im}~\boldsymbol \theta_i|, i=1, 2$ satisfies \eqref{restriction} in Lemma \ref{homosolu}.
Then  there exist complex geometric optics solutions
\begin{align*}
	u_1(\boldsymbol x)=e^{{\rm i}\boldsymbol{\theta}_1\cdot\boldsymbol x}(1+p_1(\boldsymbol x)),~\quad  u_2(\boldsymbol x)=e^{{\rm i}\boldsymbol{\theta}_2\cdot\boldsymbol x}(1+p_2(\boldsymbol x))
\end{align*}
for  equation
\[\Delta^2 u_i+\gamma\Delta u_i-k^4 u_i+q_i(\boldsymbol x)u_i=0,\quad i=1, 2 \quad \text{in}~ \Omega,\]
 respectively.

Multiplying $u_1$ by $ u_2$, we obtain
\begin{align*}
	 u_1 u_2=e^{-{\rm i}\boldsymbol \iota\cdot\boldsymbol x}(1+R(\boldsymbol x)),
\end{align*}
where the remainder \[R(\boldsymbol x)= p_1(\boldsymbol x)+p_2(\boldsymbol x)+ p_1(\boldsymbol x) p_2(\boldsymbol x).\]
Using the above inequalities  \eqref{ine2} and the inequality $x^a\leq a !e^x, a\in \mathbb Z_+,  \text{for}~ x>0 $, we have
\begin{align}
\|u_i\|^2_{L^2(\Omega)}	\leq &\|e^{{\rm i}\boldsymbol \theta_i\cdot\boldsymbol x}\|^2_{L^2(\Omega)}+\|e^{{\rm i}\boldsymbol \theta_i\cdot\boldsymbol x}p_i\|^2_{L^2(\Omega)}\leq Ce^{2\beta},\label{z1}\\
\|\partial_{x_mx_nx_sx_t}u_i\|^2_{L^2(\Omega)}
\leq& C(\sqrt{\gamma^2+4k^4}+\gamma)^4\beta^8e^{2\beta}
\leq C(\gamma^2+4k^4)^2e^{3\beta}\label{z2}.
\end{align}
Moreover, combining the inequality \eqref{ine2} and using the Cauchy-Schwartz inequality, we obtain
\begin{align}\label{z3}
\int_{\Omega}|R(\boldsymbol x)|~{\rm d}\boldsymbol x
\leq &C(\|p_1(\boldsymbol x)\|_{L^2(\Omega)}+\|p_2(\boldsymbol x)\|_{L^2(\Omega)}+\|p_1(\boldsymbol x)\|_{L^2(\Omega)}\|p_2(\boldsymbol x)\|_{L^2(\Omega)})\nonumber\\
\leq &C(\frac{2\beta+1}{\beta^2})
\leq \frac{C}{\beta}.
\end{align}
According to \eqref{ineq13} in Lemma \ref{substr},  and \eqref{z1}--\eqref{z3}, we have
\begin{align}\label{ineq14}
	|\widehat{(q_{1\iota}-q_{2\iota})}|=&\frac{1}{(2\mathcal R)^{3/2}}\bigg|\int_{[-\mathcal R, \mathcal R]^3}(q_1-q_2)(\boldsymbol x) e^{-{\rm i}\boldsymbol\iota\cdot\boldsymbol x}~{\rm d}\boldsymbol x\bigg|\nonumber\\
	=&\frac{1}{(2\mathcal R)^{3/2}}\bigg|\int_{[-\mathcal R, \mathcal R]^3}(q_1-q_2)(\boldsymbol x)u_1  u_2-(q_1-q_2)(\boldsymbol x)e^{-{\rm i}\boldsymbol\iota\cdot\boldsymbol x}R(\boldsymbol x)~{\rm d}\boldsymbol x\bigg|\nonumber\\
\leq&\frac{1}{(2\mathcal R)^{3/2}}\left(\bigg|\int_{\Omega}(q_1-q_2)(\boldsymbol x) u_1 u_2~{\rm d}\boldsymbol x\bigg|+\bigg|\int_\Omega(q_1-q_2)(\boldsymbol x)e^{-{\rm i}\boldsymbol\iota\cdot\boldsymbol x}R(\boldsymbol x)~{\rm d}\boldsymbol x\bigg|\right)\nonumber\\
\leq&C_1\big((1+|\gamma|)(\gamma^2+4k^4)^2e^{3\beta}{\rm dist}(C_{q_{1}}, C_{q_{2}})+\frac{1}{\beta}\big),
\end{align}
where the positive constant $C_1$ depends on $s, \mathcal R$ and $c_s$.

Substituting \eqref{I1} and \eqref{ineq14} into \eqref{ineq15}, we obtain
\begin{align*}
\|(q_1-q_2)(\boldsymbol x)\|^2_{L^2(\Omega)}\leq &\frac{c_s^2}{\eta^{2s}}
+C_1^2\big(\eta^3(1+|\gamma|)(\gamma^2+4k^4)^2e^{3\beta}{\rm dist}(C_{q_{1}}, C_{q_{2}})+\frac{\eta^3}{\beta}\big)^2\nonumber\\
\leq &C_2^2\big(\frac{1}{\eta^{s}}+(1+|\gamma|)(\gamma^2+4k^4)^2e^{3\beta+\eta}{\rm dist}(C_{q_{1}}, C_{q_{2}})+\frac{\eta^3}{\beta}\big)^2,
\end{align*}
where $C_2:=\max\{c_s, 6C_1\}$.

Let  $\eta:=\beta^{1/(s+3)}$ with $\beta>\beta_0+2^{s+3}$ and
\begin{align}\label{beta}
	\beta_0:=\max\{1, ~(\sqrt{\gamma^2+4k^4}+\gamma)/2, ~2C\|q(\boldsymbol x)\|_{L^\infty(\Omega)}\},
\end{align}
so that  $\eta>2$ holds.
We have
\begin{align}\label{inequlityL2}
	\|(q_1-q_2)(\boldsymbol x)\|^2_{L^2(\Omega)}
&\leq C_2^2\big(\frac{2}{\beta^{s/(s+3)}}+(1+|\gamma|)(\gamma^2+4k^4)^2e^{3\beta+\beta^{1/(s+3)}}{\rm dist}(C_{q_{1}}, C_{q_{2}})\big)^2\nonumber\\
	&\leq C_2^2\big(\frac{2}{\beta^{s/(s+3)}}+(1+|\gamma|)(\gamma^2+4k^4)^2e^{4\beta}{\rm dist}(C_{q_{1}}, C_{q_{2}})\big)^2.
\end{align}
We assume ${\rm dist}(C_{q_{1}}, C_{q_{2}})<\delta$  is sufficiently small such that
\[\delta\leq e^{-4(s+3)(\beta_0+2^{s+3})},\]
and denote
\begin{align*}
\beta:=-\frac{1}{4(s+3)}\ln ({\rm dist}(C_{q_{1}}, C_{q_{2}})),
\end{align*}
then $\beta>\beta_0$ is satisfied.
Combining with \eqref{inequlityL2}, it gives that
\begin{align*}
	\|(q_1-q_2)(\boldsymbol x)\|_{L^2(\Omega)}\leq &C_3\big((-\ln ({\rm dist}(C_{q_{1}}, C_{q_{2}})))^{-s/(s+3)}\\
	&+(1+|\gamma|)(\gamma^2+4k^4)^2{\rm dist}(C_{q_{1}}, C_{q_{2}})^{(s+2)/(s+3)}\big),
\end{align*}
where $C_3:=\max\{2(4(s+3))^{s/(s+3)}, 1\}C_2$.
The proof is completed.

\subsection*{The proof of Corollary \ref{sta2}} Using analogue analysis as Theorem \ref{stability}, we also divide the proof  into two parts: one for $\underset{|\boldsymbol\iota|>\eta}{\sum} |\widehat{(q_{1\boldsymbol\iota}-q_{2\boldsymbol\iota})}|$ and the other for $\underset{|\boldsymbol\iota|\leq\eta}{\sum} |\widehat{(q_{1\boldsymbol\iota}-q_{2\boldsymbol\iota})}|$.

For high frequency, applying the Cauchy-Schwartz inequality, one can see that
\begin{align*}
\underset{|\boldsymbol\iota|>\eta}{\sum} |\widehat{(q_{1\boldsymbol\iota}-{q}_{2\boldsymbol\iota})}|\leq &	C(\underset{|\boldsymbol\iota|>\eta}{\sum} (1+\boldsymbol \iota\cdot\boldsymbol\iota)^s|\widehat{(q_{1\boldsymbol\iota}-{q}_{2\boldsymbol\iota})}|^2)^{1/2}(\underset{|\boldsymbol\iota|>\eta}{\sum} \frac{1}{(1+\boldsymbol \iota\cdot\boldsymbol\iota)^s})^{1/2}\\
\leq &C\frac{c_s}{\eta^s}\leq \frac{C_4}{\eta^{s-3/2}},
\end{align*}
where $C_4$ depends on $c_s$.

For the  low frequency term,
\begin{align*}
	|\widehat{(q_{1\boldsymbol\iota}-q_{2\boldsymbol\iota})}|=&\frac{1}{(2\mathcal R)^{3/2}}\bigg|\int_{[-\mathcal R, \mathcal R]^3}(q_1-q_2)(\boldsymbol x) e^{-{\rm i}\boldsymbol\iota\cdot\boldsymbol x}~{\rm d}\boldsymbol x\bigg|\nonumber\\
	=&\frac{1}{(2\mathcal R)^{3/2}}\bigg|\int_{\Omega}(q_1-q_2)(\boldsymbol x)  u_1 u_2-(q_1-q_2)(\boldsymbol x)e^{-{\rm i}\boldsymbol\iota\cdot\boldsymbol x}R(\boldsymbol x)~{\rm d}\boldsymbol x\bigg|\nonumber\\
	\leq&C_5\big((1+|\gamma|)(\gamma^2+4k^4)^2e^{3\beta}{\rm dist}(C_{q_{1}}, C_{q_{2}})+\frac{1}{\beta}\big),
\end{align*}
where $C_5$ depends on $s, \mathcal R$ and $c_s$.
Combining above estimates, we get
\begin{align*}
	\|(q_1-q_2)(\boldsymbol x)\|_{L^\infty(\Omega)}
\leq &\frac{C_4}{\eta^{s-3/2}}+C_5(\eta^3(1+|\gamma|)(\gamma^2+4k^4)^2e^{3\beta}{\rm dist}(C_{q_{1}}, C_{q_{2}})+\frac{\eta^3}{\beta})\\
\leq &C_6 \big(\frac{1}{\eta^{s-3/2}}+(1+|\gamma|)(\gamma^2+4k^4)^2e^{3\beta+\eta}{\rm dist}(C_{q_{1}}, C_{q_{2}})+\frac{\eta^3}{\beta}\big),
\end{align*}
where $C_6:=\max\{C_4, 6C_5\}$.

Define $\eta:=\beta^{2/(2s+3)}$ with $\beta >\beta_0+2^{s+3}$, and $\beta_0$ satisfies \eqref{beta}.
Then, we have
\begin{align}\label{E6}
	\|(q_1-q_2)(\boldsymbol x)\|_{L^\infty(\Omega)}
\leq &C_6 \big(\frac{2}{\beta^{(2s-3)/(2s+3)}}+(1+|\gamma|)(\gamma^2+4k^4)^2e^{4\beta}{\rm dist}(C_{q_{1}}, C_{q_{2}})\big).
\end{align}
We denote
\begin{align}\label{beta2}
\beta:=-\frac{1}{4(2s+3)}\ln({\rm dist}(C_{q_{1}}, C_{q_{2}})),
\end{align}
with  ${\rm dist}(C_{q_{1}}, C_{q_{2}})<\delta\leq e^{-4(2s+3)(\beta_0+2^{s+3})}$, which implies $\beta >\beta_0+2^{s+3}$ and $\eta>2$.

Substituting \eqref{beta2} into \eqref{E6},  a direct  calculation yields
\begin{align*}
	\|(q_1-q_2)(\boldsymbol x)\|_{L^\infty(\Omega)}
	\leq &C_7\big((-\ln({\rm dist}(C_{q_{1}}, C_{q_{2}})))^{-(2s-3)/(2s+3)}\\
	&+(1+|\gamma|)(\gamma^2+4k^4)^2{\rm dist}(C_{q_{1}}, C_{q_{2}})^{{(2s+2)/(2s+3)}}\big).
\end{align*}
Here $C_7:=\max\{2(4(2s+3))^{s-3/2}, 1\}C_6$.
The proof is completed.
\begin{rema}
In order to satisfy Theorem \ref{stability}	and Corollary \ref{sta2}, ${\rm dist}(C_{q_{1}}, C_{q_{2}})$ should  be sufficiently small such that
\begin{align*}
	{\rm dist}(&C_{q_{1}}, C_{q_{2}})\leq  e^{-4(2s+3)(\beta_0+2^{s+3})}.
\end{align*}
\end{rema}

Finally, we will improve previous estimates.
\subsection*{The proof of Theorem \ref{reducesta3}}
We refer to Theorem \ref{stability} and Corollary \ref{sta2}.
It follows from the Cauchy-Schwarz inequality that
\begin{align}\label{y1}
\underset{|\boldsymbol\iota|>\eta}{\sum} |\widehat{(q_{1\boldsymbol\iota}-q_{2\boldsymbol\iota})}|\leq \underset{|\boldsymbol\iota|>\eta}{\sum}\frac{1}{(1+\boldsymbol \iota\cdot\boldsymbol \iota)^{m/2}} (1+\boldsymbol \iota\cdot \boldsymbol \iota)^{m/2}|\widehat{(q_{1\boldsymbol\iota}-{q}_{2\boldsymbol\iota})}|\leq\frac{c_m}{\eta^{m-3}}
\end{align}
under the assumption $\eta>2$.

For the low-frequency term, by the estimate \eqref{ineq14} with changing the assumption $\beta>\beta_0+16C$ leads to
\begin{align}\label{low}
|\widehat{(q_{1\boldsymbol\iota}-q_{2\boldsymbol\iota})}|=&\frac{1}{(2\mathcal R)^{3/2}} \bigg|\int_{[-\mathcal R, \mathcal R]^3}(q_1-q_2)(\boldsymbol x) e^{-{\rm i}\boldsymbol\iota\cdot\boldsymbol x}~{\rm d}\boldsymbol x\bigg|\nonumber\\
	=&\frac{1}{(2\mathcal R)^{3/2}}\bigg|\int_{\Omega}(q_1-q_2)(\boldsymbol x) u_1 u_2-(q_1-q_2)(\boldsymbol x)e^{-{\rm i}\boldsymbol\iota\cdot\boldsymbol x}R(\boldsymbol x)~{\rm d}\boldsymbol x\bigg|\nonumber\\
	\leq&C_8\big((1+|\gamma|)(\gamma^2+4k^4)^2e^{3\beta}{\rm dist}(C_{q_{1}}, C_{q_{2}})+\frac{\|(q_1-q_2)(\boldsymbol x)\|_{L^\infty(\Omega)}}{\beta}\big),
\end{align}
where $C_8$ is a suitable constant and depends on $m, \mathcal R$ and $c_m$.
Combining \eqref{y1} and \eqref{low}, and taking  $\eta:=(\frac{\beta}{2C_8})^{1/3}$,  we have
\begin{align}\label{E3}
	\|(q_1-q_2)(\boldsymbol x)\|_{L^\infty(\Omega)}
\leq C_{9}\big(&\frac{1}{\beta^{(m-3)/3}}+(1+|\gamma|)(\gamma^2+4k^4)^2e^{4\beta}{\rm dist}(C_{q_{1}}, C_{q_{2}})\big)\nonumber\\
&+\frac{1}{2}\|(q_1-q_2)(\boldsymbol x)\|_{L^\infty(\Omega)},
\end{align}
where $C_{9}:=\max\{(2C_8)^{(m-3)/3}c_m, 6C_8\}$.

Let ${\rm dist}(C_{q_{1}}, C_{q_{2}})< \delta$ with  $\delta\leq e^{-12(\beta_0+16C_8)}$ and choosing
\begin{align}\label{beta3}
\beta:=-\frac{1}{12}\ln({\rm dist}(C_{q_{1}}, C_{q_{2}})),
\end{align}
so that  $\beta>\beta_0+16C_8$ and $\eta>2$ holds.

Substituting \eqref{beta3} into \eqref{E3}, we obtain
\begin{align*}
	\frac{1}{2}\|(q_1-q_2)(\boldsymbol x)\|_{L^\infty(\Omega)}
	\leq &C_{10}\big((-\ln({\rm dist}(C_{q_{1}}, C_{q_{2}})))^{-(m-3)/3}\\
	&+(1+|\gamma|)(\gamma^2+4k^4)^2{\rm dist}(C_{q_{1}}, C_{q_{2}})^{2/3}\big),
\end{align*}
where $C_{10}:=\max\{12^{(m-3)/3}, 1\}C_{9}$.
The proof is completed.
\begin{rema}
The above stability results can be derived to higher-dimensional spaces, that is, $d\geq 3$.
\end{rema}

\section*{Acknowledgment}
The authors would like to thank the anonymous reviewers
for their constructive suggestions and comments on improving the presentation of
the paper.

{\bf Conflict of interest }

This work does not have any conflicts of interest.


\begin{thebibliography}{10}

\bibitem{2015Agranovich}
M.~S. Agranovich.
\newblock {\em Sobolev spaces, their generalizations and elliptic problems in
  smooth and {L}ipschitz domains}.
\newblock Springer Monographs in Mathematics. Springer, Cham, 2015.

\bibitem{Yang2017}
Y.~M. Assylbekov and Y.~Yang.
\newblock Determining the first order perturbation of a polyharmonic operator
  on admissible manifolds.
\newblock {\em J. Differential Equations}, 262(1):590--614, 2017.

\bibitem{Boulenger2017}
T.~Boulenger and E.~Lenzmann.
\newblock Blowup for biharmonic {NLS}.
\newblock {\em Ann. Sci. \'{E}c. Norm. Sup\'{e}r. (4)}, 50(3):503--544, 2017.

\bibitem{Deng20192}
Y.~Deng, J.~Li, and H.~Liu.
\newblock On identifying magnetized anomalies using geomagnetic monitoring.
\newblock {\em Arch. Ration. Mech. Anal.}, 231(1):153--187, 2019.

\bibitem{Deng2019}
Y.~Deng, H.~Liu, and G.~Uhlmann.
\newblock On an inverse boundary problem arising in brain imaging.
\newblock {\em J. Differential Equations}, 267(4):2471--2502, 2019.

\bibitem{feldman2019calderon}
J.~Feldman, M.~Salo, and G.~Uhlmann.
\newblock Calder{\'o}n problem: An introduction to inverse problems.
\newblock {\em Preliminary notes on the book in preparation}, 2019.

\bibitem{Fibich2002}
G.~Fibich, B.~Ilan, and G.~Papanicolaou.
\newblock Self-focusing with fourth-order dispersion.
\newblock {\em SIAM J. Appl. Math.}, 62(4):1437--1462, 2002.

 \bibitem{GLL2023}
Y.~Gao, H.~Liu, and Y.~Liu.
\newblock On an inverse problem for the plate equation with passive measurement.
\newblock {\em SIAM J. Appl. Math.}, 83(3):1196--1214, 2023.

\bibitem{G2010}
F.~Gazzola, H.-C. Grunau, and G.~Sweers.
\newblock {\em Polyharmonic boundary value problems}, volume 1991 of {\em
  Lecture Notes in Mathematics}.
\newblock Springer-Verlag, Berlin, 2010.


 \bibitem{Hahner1996}
P.~H\"{a}hner.
\newblock A periodic {F}addeev-type solution operator.
\newblock {\em J. Differential Equations}, 128(1):300--308, 1996.

\bibitem{Hahner2001}
P.~H\"{a}hner and T.~Hohage.
\newblock New stability estimates for the inverse acoustic inhomogeneous medium
  problem and applications.
\newblock {\em SIAM J. Math. Anal.}, 33(3):670--685, 2001.

\bibitem{isakov2011}
V.~Isakov.
\newblock Increasing stability for the {S}chr\"{o}dinger potential from the
  {D}irichlet-to-{N}eumann map.
\newblock {\em Discrete Contin. Dyn. Syst. Ser. S}, 4(3):631--640, 2011.

\bibitem{isa2016}
V.~Isakov, R.-Y. Lai, and J.-N. Wang.
\newblock Increasing stability for the conductivity and attenuation
  coefficients.
\newblock {\em SIAM J. Math. Anal.}, 48(1):569--594, 2016.

\bibitem{karpman}
V.~I. Karpman.
\newblock Stabilization of soliton instabilities by higher order dispersion:
  {K}d{V}-type equations.
\newblock {\em Phys. Lett. A}, 210(1-2):77--84, 1996.

\bibitem{karp2}
V.~I. Karpman and A.~G. Shagalov.
\newblock Stability of solitons described by nonlinear {S}chr\"{o}dinger-type
  equations with higher-order dispersion.
\newblock {\em Phys. D}, 144(1-2):194--210, 2000.

\bibitem{knox2020}
C.~Knox and A.~Moradifam.
\newblock Determining both the source of a wave and its speed in a medium from
  boundary measurements.
\newblock {\em Inverse Problems}, 36(2):025002, 15, 2020.

\bibitem{wang2021}
P.-Z. Kow, G.~Uhlmann, and J.-N. Wang.
\newblock Optimality of increasing stability for an inverse boundary value
  problem.
\newblock {\em SIAM J. Math. Anal.}, 53(6):7062--7080, 2021.

\bibitem{Gunther2012}
K.~Krupchyk, M.~Lassas, and G.~Uhlmann.
\newblock Determining a first order perturbation of the biharmonic operator by
  partial boundary measurements.
\newblock {\em J. Funct. Anal.}, 262(4):1781--1801, 2012.

\bibitem{Katsi2014}
K.~Krupchyk, M.~Lassas, and G.~Uhlmann.
\newblock Inverse boundary value problems for the perturbed polyharmonic
  operator.
\newblock {\em Trans. Amer. Math. Soc.}, 366(1):95--112, 2014.

\bibitem{Lijingzhi2021}
J.~Li, H.~Liu, and S.~Ma.
\newblock Determining a random {S}chr\"{o}dinger operator: both potential and
  source are random.
\newblock {\em Comm. Math. Phys.}, 381(2):527--556, 2021.

\bibitem{li2021stability}
P.~Li, X.~Yao, and Y.~Zhao.
\newblock Stability for an inverse source problem of the biharmonic operator.
\newblock {\em SIAM J. Appl. Math.}, 81(6):2503--2525, 2021.

\bibitem{Li2021}
P.~Li, X.~Yao, and Y.~Zhao.
\newblock Stability of an inverse source problem for the damped biharmonic
  plate equation.
\newblock {\em Inverse Problems}, 37(8):Paper No. 085003, 19, 2021.

\bibitem{Liuboya2020}
B.~Liu.
\newblock Stability estimates in a partial data inverse boundary value problem
  for biharmonic operators at high frequencies.
\newblock {\em Inverse Probl. Imaging}, 14(5):783--796, 2020.

\bibitem{Liu2015}
H.~Liu and G.~Uhlmann.
\newblock Determining both sound speed and internal source in thermo- and
  photo-acoustic tomography.
\newblock {\em Inverse Problems}, 31(10):105005, 10, 2015.

\bibitem{liu2016}
S.~Liu and L.~Oksanen.
\newblock A {L}ipschitz stable reconstruction formula for the inverse problem
  for the wave equation.
\newblock {\em Trans. Amer. Math. Soc.}, 368(1):319--335, 2016.

\bibitem{Liu20214}
X.~Liu and T.~Zhang.
\newblock The {C}auchy problem for the fourth-order {S}chr\"{o}dinger equation
  in {$H^{s}$}.
\newblock {\em J. Math. Phys.}, 62(7):Paper No. 071501, 20, 2021.

\bibitem{Miao2009}
C.~Miao, G.~Xu, and L.~Zhao.
\newblock Global well-posedness and scattering for the focusing energy-critical
  nonlinear {S}chr\"{o}dinger equations of fourth order in the radial case.
\newblock {\em J. Differential Equations}, 246(9):3715--3749, 2009.

\bibitem{nagayasu2013}
S.~Nagayasu, G.~Uhlmann, and J.-N. Wang.
\newblock Increasing stability in an inverse problem for the acoustic equation.
\newblock {\em Inverse Problems}, 29(2):025012, 11, 2013.

\bibitem{Pausader2007}
B.~Pausader.
\newblock Global well-posedness for energy critical fourth-order
  {S}chr\"{o}dinger equations in the radial case.
\newblock {\em Dyn. Partial Differ. Equ.}, 4(3):197--225, 2007.

\bibitem{Pausader2009}
B.~Pausader.
\newblock The cubic fourth-order {S}chr\"{o}dinger equation.
\newblock {\em J. Funct. Anal.}, 256(8):2473--2517, 2009.

\bibitem{Yang2014}
Y.~Yang.
\newblock Determining the first order perturbation of a bi-harmonic operator on
  bounded and unbounded domains from partial data.
\newblock {\em J. Differential Equations}, 257(10):3607--3639, 2014.

\end{thebibliography}

\end{document}